\documentclass[]{article}

\usepackage{amssymb}
\usepackage{amsmath}
\usepackage{algpseudocode}
\usepackage{url,here}
\usepackage[backref,colorlinks=true]{hyperref}
\usepackage{multirow}
\usepackage{algorithm}
\allowdisplaybreaks

\begin{document}
\newenvironment {proof}{{\noindent\bf Proof.}}{\hfill $\Box$ \medskip}

\newtheorem{theorem}{Theorem}[section]
\newtheorem{lemma}[theorem]{Lemma}
\newtheorem{condition}[theorem]{Condition}
\newtheorem{proposition}[theorem]{Proposition}
\newtheorem{remark}[theorem]{Remark}
\newtheorem{definition}[theorem]{Definition}
\newtheorem{hypothesis}[theorem]{Hypothesis}
\newtheorem{corollary}[theorem]{Corollary}
\newtheorem{example}[theorem]{Example}
\newtheorem{descript}[theorem]{Description}
\newtheorem{assumption}[theorem]{Assumption}

\newcommand{\ba}{\begin{align}}
\newcommand{\ea}{\end{align}}

\def\P{\mathbb{P}}
\def\R{\mathbb{R}}
\def\E{\mathbb{E}}
\def\N{\mathbb{N}}
\def\Z{\mathbb{Z}}

\renewcommand {\theequation}{\arabic{section}.\arabic{equation}}
\def \non{{\nonumber}}
\def \hat{\widehat}
\def \tilde{\widetilde}
\def \bar{\overline}

\def\ind{{\mathchoice {\rm 1\mskip-4mu l} {\rm 1\mskip-4mu l}
{\rm 1\mskip-4.5mu l} {\rm 1\mskip-5mu l}}}

\title{\Large\ {\bf Unbiased estimation of second-order parameter sensitivities for stochastic reaction networks }}

\author{Ankit Gupta and Mustafa Khammash\\
}
\date{\today}
\maketitle
\begin{abstract}
This paper deals with the problem of estimating second-order parameter sensitivities for stochastic reaction networks, where the reaction dynamics is modeled as a continuous time Markov chain over a discrete state space. Estimation of such second-order sensitivities (the Hessian) is necessary for implementing the Newton-Raphson scheme for optimization over the parameter space. To perform this estimation, Wolf and Anderson \cite{Beth} have proposed an efficient finite-difference method, that uses a coupling of perturbed processes to reduce the estimator variance. The aim of this paper is to illustrate that the same coupling can be exploited to derive an exact representation for second-order parameter sensitivity. Furthermore with this representation one can construct an unbiased estimator which is easy to implement. The ideas contained in this paper are extensions of the ideas presented in \cite{Gupta,Gupta2} in the context of first-order parameter sensitivity estimation.
\end{abstract}

\medskip
\setcounter{equation}{0}

\section{Problem Definition}

Consider a reaction network with $d$ species and $K$ reactions whose \emph{stoichiometric} vectors are given by $\zeta_1,\dots,\zeta_K$. We assume that propensities of various reactions depend on a vector of $p$ parameters $\theta = (\theta_1,\dots,\theta_p)$, which may represent systems's parameters such as, reaction rate constants, hill coefficients, cell volume etc. When the state is $x$, the $k$-th reaction fires at rate $\lambda_k(x,\theta)$ and it moves the state to $(x+\zeta_k)$. In the stochastic setting, the reaction dynamics can be represented by a Markov process whose generator is
\begin{align*}
\mathbb{A}_{\theta} f(x) = \sum_{k=1}^K \lambda_k(x,\theta) \Delta_{\zeta_k} f(x),
\end{align*}
where $\Delta_{\zeta_k} f(x) = f(x+\zeta_k)-f(x)$.

Let $(X_\theta(t))_{t \geq 0}$ be a process with generator $\mathbb{A}_{\theta}$. Then its random time change representation (see \cite{EK}) is given by
\begin{align*}
X_\theta(t) = X_\theta(0) + \sum_{k=1}^K Y_k\left( \int_{0}^{t} \lambda_k( X_\theta(s),\theta  )ds \right) \zeta_k,
\end{align*}
where $\{ Y_k :  k =1,\dots,K\}$ is a family of independent unit rate Poisson processes. For any function $f : \N^d_0 \to \R$ that expresses an output of interest, define
\begin{align*}
\Psi_\theta(x,f,t) = \E\left( f( X_\theta(t) )  \vert X_\theta(0) =x  \right).
\end{align*}
Then for any $i=1,\dots,p$, the \emph{first-order} sensitivity
\begin{align}
\label{def:firstsens}
S^{(i)}_{\theta} (x, f, t) = \frac{  \partial \Psi_\theta(x,f,t)  }{  \partial \theta_i  }
\end{align}
measures how sensitive the expected value of the output at time $t$, $\E( f( X_\theta(t) )$, is to small changes in parameter $\theta_i$. 
Many methods exist in the literature to estimate $S^{(i)}_{\theta} (x, f, t) $ (see \cite{KSR1,KSR2,DA,Gupta,Gupta2,Gir}). In this paper we deal with the problem of estimating the \emph{second-order} sensitivity
\begin{align*}
S^{(i,j)}_{\theta} (x, f, t) = \frac{  \partial^2 \Psi_\theta(x,f,t)  }{  \partial \theta_i \partial \theta_j  },
\end{align*}
for any $i,j \in \{1,\dots,p\}$. Such a quantity measures the local curvature of the mapping $\theta \mapsto \E ( f(X_\theta(t)) )$ and this information is useful in implementing optimization schemes such as the Newton Raphson method.

In \cite{Beth}, the authors estimate $S^{(i,j)}_{\theta} (x, f, t) $ using a finite-difference approximation of the form
\begin{align}
\label{daz_fd}
S^{(i,j)}_{\theta} (x, f, t) \approx \E\left( \frac{ f( X_{\theta_1^\epsilon}(t) ) - f( X_{\theta_2^\epsilon}(t) ) - f( X_{\theta_3^\epsilon}(t) ) + f( X_{\theta_4^\epsilon}(t) )   }{\epsilon^2}  \right),
\end{align}
for a small $\epsilon$, where each $X_{\theta_l^\epsilon}$ is a process with generator $\mathbb{A}_{\theta_l^\epsilon} $ and initial state $x$, and $\theta_l^\epsilon$-s denote perturbations of the parameter $\theta$ defined as follows:\footnote{Here $e_i$ denotes the vector $(0,\dots,0,1,0,\dots) \in \R^q$ where the $1$ is at the $i$-th location.}
\begin{align*}
\theta_1^\epsilon = \theta +(e_i+e_j) \epsilon, \quad  \theta_2^\epsilon = \theta +e_i \epsilon, \quad \theta_3^\epsilon = \theta +e_j \epsilon \quad \textnormal{and} \quad \theta_1^\epsilon = \theta.
\end{align*}
Moreover the processes $X_{\theta_1^\epsilon},X_{\theta_2^\epsilon},X_{\theta_3^\epsilon}$ and $X_{\theta_4^\epsilon}$, are intelligently coupled to lower the variance of the associated estimator. The main drawback of finite-difference approximation is that it introduces a \emph{bias} in the estimate and generally the size or even the sign of this bias is unknown, which can cause problems in certain applications.

Interestingly, the coupling described in \cite{Beth} can be used to derive an exact formula for $S^{(i,j)}_{\theta} (x, f, t)$ by extending the ideas presented in \cite{Gupta} in the context of first-order sensitivity. The advantage of such a formula is that it allows one to construct an efficient unbiased estimator for $S^{(i,j)}_{\theta} (x, f, t)$, in the same way as the formula for the first-order sensitivity $S^{(i)}_{\theta} (x, f, t)$ in \cite{Gupta} was used in \cite{Gupta2} to devise an unbiased estimator for this quantity. In Section \ref{mainresult} we present our main result that expresses $S^{(i,j)}_{\theta} (x, f, t)$ as the expectation of a certain random variable. In Section \ref{algo} we describe how this result can be used for obtaining unbiased estimates of $S^{(i,j)}_{\theta} (x, f, t)$ in an efficient way.

\section{Main Result} \label{mainresult}
From now on, let $\lambda_0(x,\theta)$ be the function denoting the sum of propensities
\begin{align*}
\lambda_0(x,\theta) = \sum_{k=1}^K \lambda_k(x,\theta).
\end{align*}

\begin{theorem}
\label{th:mainresult}
Suppose $(X_\theta(t))_{t \geq 0}$ is the Markov process with generator $\mathbb{A}_\theta$ and initial state $x_0$. Let $\sigma_l$ be the $l$-th jump time of the process for $l=0,1,\dots$. Then for any function $f : \N^d_0 \to \R$ and $t \geq 0$, $S^{(i,j)}_{\theta} (x_0, f, t)  = \E\left( s^{(i,j)}_\theta (x_0,f,t) \right)$ where
\begin{align}
\label{defnsij}
& s^{(i,j)}_\theta (x_0,f,t)\\ &= \sum_{k=1}^K \left[ \int_{0}^t \frac{ \partial^2 \lambda_k  ( X_\theta(s) ,\theta ) }{ \partial \theta_i \partial \theta_j }  \Delta_{\zeta_k} f ( X_\theta(s) ) ds  \right.  \notag\\ 
& \left.   +  \sum_{l=0, \sigma_l<t}^\infty \frac{ \partial \lambda_k( X_\theta(\sigma_l) ,\theta  ) }{ \partial \theta_i }  \left( \int_{0}^{t - \sigma_l} 
(S^{(j)}_{\theta} (X_\theta(\sigma_l) + \zeta_k, f, t - \sigma_l -s)  - S^{(j)}_{\theta} (X_\theta(\sigma_l), f, t - \sigma_l - s) ) e^{-\lambda_0( X_\theta(\sigma_l) ,\theta  ) s   } ds \right) \right. \notag \\ & \left.   +  \sum_{l=0, \sigma_l<t}^\infty \frac{ \partial \lambda_k( X_\theta(\sigma_l) ,\theta  ) }{ \partial \theta_j }  \left( \int_{0}^{t - \sigma_l} 
(S^{(i)}_{\theta} (X_\theta(\sigma_l) + \zeta_k, f, t -\sigma_l -s)  - S^{(i)}_{\theta} (X_\theta(\sigma_l), f, t-\sigma_l -s) ) e^{-\lambda_0( X_\theta(\sigma_l) ,\theta  ) s   } ds \right) \right. \notag \\ & \left.   +  \sum_{l=0, \sigma_l<t}^\infty\frac{ \partial^2 \lambda_k  ( X_\theta(\sigma_l) ,\theta ) }{ \partial \theta_i \partial \theta_j }    \left( \int_{0}^{t - \sigma_l} 
( \Psi_\theta (X_\theta(\sigma_l) + \zeta_k, f, t - \sigma_l -s)  - \Psi_\theta(X_\theta(\sigma_l), f, t - \sigma_l-s)  \right. \right. \notag \\ &  \left. \left. \qquad \qquad  \qquad \qquad  \qquad  \qquad -\Delta_{\zeta_k} f(X_\theta(\sigma_l) )  ) e^{-\lambda_0( X_\theta(\sigma_l) ,\theta  ) s } ds \right) \right], \notag
\end{align}
\end{theorem}
\begin{proof}
The proof follows by a simple extension of the ideas presented in \cite{Gupta}. We start with the finite-difference approximation of the form \eqref{daz_fd}, where the processes $X_{\theta_1^\epsilon},X_{\theta_2^\epsilon},X_{\theta_3^\epsilon}$ and $X_{\theta_4^\epsilon}$ are coupled in the same way as described in \cite{Beth}. Using Dynkin's formula and exploiting the coupling, we can pass to the limit $\epsilon \to 0$ and prove the relation given above. The details of the proof shall be provided elsewhere. 
\end{proof}

Observe that Theorem \ref{th:mainresult} expresses the second-order sensitivity as the expectation of a random variable which only involves first-order sensitivities and expectations of the underlying Markov process. Since many efficient methods exist to estimate such first-order sensitivities and expectations, one can hope to use Theorem \ref{th:mainresult} to construct an efficient estimator for the second-order sensitivity. However the main difficulty is that one has to estimate a ``new" quantity (first-order sensitivity and/or expectation) at each jump time in the observation time period $[0,t]$. This can be very cumbersome as the number of jumps can be very high. However using the ideas in \cite{Gupta2} we can get around this problem and only estimate these ``new" quantities at a \emph{small} number of jump times, and still achieve an unbiased estimate for the second-order sensitivity. We describe this approach in the next section.

\section{Algorithm} \label{algo}

We use the same notation as in Theorem \ref{th:mainresult}. Let $(X_\theta(t))_{t \geq 0}$ be the Markov process with generator $\mathbb{A}_\theta$, initial state $x_0$, and let $\sigma_l$ be the $l$-th jump time of the process for $l=0,1,\dots$. The total number of jumps until time $t$ is given by the random variable
\begin{align}
\label{defn_eta}
\eta_t = \max \{ i \geq 0 :  \sigma_i < t \}.
\end{align}
For simplicity we assume that there are no absorbing states (that is, $\lambda_0(x,\theta) >0$ for all $x \in \N^d_0$). For each $l=0,\dots,\eta_l$ let $\gamma_l$ be an independent exponentially distributed random variable with rate $\lambda_0( X_\theta( \sigma_l ),\theta )$ and define
\begin{align}
\label{defn_Gammal}
\Gamma_l =   \left\{
\begin{array}{cc}
1 & \textnormal{ if } \gamma_l < (t - \sigma_l)  \\
0 & \textnormal{ otherwise}
\end{array}
\right\}. 
\end{align}
For each $k\in \{1,\dots,K\}$ and $q \in \{ i,j\}$ let $ \beta^{(q)}_{kl} $ be given by
\begin{align*}
 \beta^{(q)}_{kl} = \textnormal{Sign}\left(  \frac{ \partial \lambda_k( X_\theta( \sigma_l) ,\theta ) }{  \partial \theta_q }   \right) \  \textnormal{ where } \ \textnormal{Sign}(x)= \left\{
\begin{array}{cc}
1   &   \textnormal{ if }  x > 0  \\
-1   &  \textnormal{ if }   x < 0  \\
0 &    \textnormal{ if } x= 0
\end{array}
\right\}.
\end{align*}
Similarly let
\begin{align*}
 \beta^{(i,j)}_{kl} = \textnormal{Sign}\left(  \frac{ \partial^2 \lambda_k( X_\theta( \sigma_l) ,\theta ) }{  \partial \theta_i \theta_j }   \right) .
\end{align*}

Now we choose a normalizing constant $c > 0$, which specifies the ``density" of jump times at which we estimate a new quantity of the form $ S^{(i)}_{\theta}(\cdot),  S^{(j)}_{\theta}(\cdot)$ or $\Psi_\theta(\cdot)$. More details on the role of $c$ and how it can be chosen can be found in \cite{Gupta2}. Please note that the estimator we construct will remain unbiased for any choice of $c$, but its variance may vary. For each $q \in \{ i,j\}$, if $ \beta^{(q)}_{kl}  \neq 0$ and $\Gamma_l = 1$, then let $\rho^{(q)}_{kl}$ be an independent $\N_0$-valued random variable whose distribution is Poisson with parameter
\begin{align}
\label{rateofpoisoon}
\frac{c}{ \lambda_0( X_\theta(\sigma_l) ,\theta ) }  \left|  \frac{ \partial \lambda_k (X_\theta(\sigma_l) ,\theta ) }{ \partial \theta_q }  \right|.
\end{align}
Similarly if $ \beta^{(i,j)}_{kl}  \neq 0$ and $\Gamma_l = 1$, then let $\rho^{(i,j)}_{kl}$ be an independent $\N_0$-valued random variable whose distribution is Poisson with parameter
\begin{align}
\label{rateofpoisoon}
\frac{c}{ \lambda_0( X_\theta(\sigma_l) ,\theta ) }  \left|  \frac{ \partial^2 \lambda_k (X_\theta(\sigma_l) ,\theta ) }{ \partial \theta_i \partial \theta_j }  \right|.
\end{align}
Let
\begin{align*}
\Delta t_l  = \left\{
\begin{array}{cc}
 (\sigma_{l+1} - \sigma_l) & \textrm{ for } l = 0,\dots,\eta_t-1  \\
 (T - \sigma_\eta) & \textrm{ for } l = \eta_t  \\
\end{array}\right\}
\end{align*}
and define
\begin{align}
\label{expr:sthetahat}
\hat{s}^{(i,j)}_\theta (x_0,f,t)  &= \sum_{k=1}^K \sum_{l = 0}^{\eta_t  }  \left[ 
\frac{ \partial^2 \lambda_k  ( X_\theta(\sigma_l) ,\theta ) }{ \partial \theta_i \partial \theta_j }  \Delta_{\zeta_k}f( X_\theta(\sigma_l)  )\left( \Delta t_l - \frac{ \Gamma_l }{ \lambda_0( X_\theta(\sigma_l) ,\theta  ) } \right) \right. \\ & \left.
+
\frac{1}{c}  \Gamma_l  \left(     \beta^{(i)}_{kl} \rho^{(i)}_{kl} \hat{S}^{(i)}_{kl} +  \beta^{(j)}_{kl} \rho^{(j)}_{kl} \hat{S}^{(j)}_{kl}   +   \beta^{(i,j)}_{kl} \rho^{(i,j)}_{kl} \hat{D}_{kl} \right)    \right],  \notag
\end{align}
where the construction of random variables $\hat{S}^{(i)}_{kl}$,  $\hat{S}^{(j)}_{kl}$ and $\hat{D}_{kl}$ is described below.

Let $(Z_1(t))_{t \geq 0}$ and $(Z_2(t))_{t \geq 0}$ be two processes with random time change representations given by:
\begin{align*}
Z_1(t) &=( X_\theta(\sigma_l)+ \zeta_k) + \sum_{k = 1}^K \hat{Y}_k\left( \int_{0}^{t}  \lambda_k( Z_1(s) ,\theta) \wedge \lambda_k( Z_2(s) ,\theta) ds \right)\zeta_k \\
&+  \sum_{k = 1}^K \hat{Y}^{(1)}_k\left( \int_{0}^{t} \left(\lambda_k( Z_1(s) ,\theta) -  \lambda_k( Z_1(s) ,\theta) \wedge \lambda_k( Z_2(s) ,\theta) \right) ds \right) \zeta_k \\  \textnormal{ and } \  
Z_2(t) &=X_\theta(\sigma_l)+ \sum_{k = 1}^K \hat{Y}_k\left( \int_{0}^{t}  \lambda_k( Z_1(s) ,\theta) \wedge \lambda_k( Z_2(s) ,\theta) ds \right)\zeta_k \\
&+  \sum_{k = 1}^K \hat{Y}^{(2)}_k\left( \int_{0}^{t} \left(\lambda_k( Z_2(s) ,\theta) -  \lambda_k( Z_1(s) ,\theta) \wedge \lambda_k( Z_2(s) ,\theta) \right) ds \right) \zeta_k,
\end{align*}
where $\{\hat{Y}_k, \hat{Y}^{(1)}_k,\hat{Y}^{(2)}_k : k =1,\dots,K\}$ is an independent family of unit rate Poisson processes. Let
\begin{align}
\label{defn_dkl}
\hat{D}_{kl} = f( Z_1( t - \sigma_l - \gamma_l ) ) - f( Z_2( t - \sigma_l - \gamma_l ) ),
\end{align}
then we must have
\begin{align}
\label{unbpr1}
\E\left( \hat{D}_{kl} \vert  \mathcal{F}_t\right) = \Psi_{\theta} (X_\theta(\sigma_l) + \zeta_k, f, t - \sigma_l - \gamma_l)  - \Psi_{\theta} (X_\theta(\sigma_l), f, t - \sigma_l - \gamma_l),
\end{align}
where $\{ \mathcal{F}_s\}_{s \geq 0}$ is the filtration generated by the process $(X_\theta(t))_{t \geq 0}$. 

We now construct $\hat{S}^{(q)}_{kl}$ for each $q \in \{i,j\}$. Using the \emph{Poisson Path Algorithm}(PPA) given in \cite{Gupta2}, with the underlying Markov process as $Z_1$, we can generate one realization of a random variable $\hat{s}^{(q)}_1$ whose expectation is
\begin{align*}
\E\left(  \hat{s}^{(q)}_1 \right) = S^{(q)}_{\theta} (X_\theta(\sigma_l) + \zeta_k, f, t - \sigma_l - \gamma_l). 
\end{align*}
Similarly using PPA with the underlying Markov process as $Z_2$, we can generate one realization of another random variable $\hat{s}^{(q)}_2$ whose expectation is
\begin{align*}
\E\left(  \hat{s}^{(q)}_2 \right) = S^{(q)}_{\theta} (X_\theta(\sigma_l), f, t - \sigma_l - \gamma_l). 
\end{align*}
Defining
\begin{align}
\label{defn_sqkl}
\hat{S}^{(q)}_{kl} =  \hat{s}^{(q)}_1 - \hat{s}^{(q)}_2,
\end{align}
we must have that
\begin{align}
\label{unbpr2}
\E\left( \hat{S}^{(q)}_{kl} \vert  \mathcal{F}_t\right) = S^{(q)}_{\theta} (X_\theta(\sigma_l) + \zeta_k, f, t - \sigma_l - \gamma_l)  - S^{(q)}_{\theta} (X_\theta(\sigma_l), f, t - \sigma_l - \gamma_l).
\end{align}

Using relations \eqref{unbpr1} and \eqref{unbpr2} one can show using a simple conditioning argument that
\begin{align*}
\E\left( s^{(i,j)}_\theta (x_0,f,t ) \right) = \E\left( \hat{s}^{(i,j)}_\theta (x_0,f,t)  \right),
\end{align*}
where $ s^{(i,j)}_\theta (x_0,f,t ) $ and $\hat{s}^{(i,j)}_\theta (x_0,f,t) $ are defined by \eqref{defnsij} and \eqref{expr:sthetahat} respectively. Hence Theorem \ref{th:mainresult} guarantees that by generating realizations of the random variable $\hat{s}^{(i,j)}_\theta (x_0,f,t)$, we can obtain an unbiased estimate for second-order parameter  sensitivity $S^{(i,j)}_{\theta} (x_0, f, t) $.

 A single realization of the random variable $\hat{s}^{(i,j)}_\theta (x_0,f,t)$ (given by \eqref{expr:sthetahat}) can be computed using \\$\Call{GenerateSample}{x_0,t,c}$ (Algorithm \ref{gensensvalue}). This method simulates the process $X_{\theta}$ according to SSA and at each state $x$ and jump time $s=\sigma_l$, the following happens:
\begin{itemize}
\item The exponential random variable $\gamma$ (where $\gamma = \gamma_l$ in \eqref{expr:sthetahat}) is generated and the corresponding $\Gamma_l$ (see \eqref{defn_Gammal}) is calculated. 
\item If $\gamma < (t-s)$ then for each $k=1,\dots,K$ such that either $\partial \lambda_k(x,\theta)/ \partial \theta_i , \partial \lambda_k(x,\theta)/ \partial \theta_j $ or $\partial^2 \lambda_k(x,\theta)/ \partial \theta_i \partial \theta_j $ is non-zero, we generate the appropriate Poisson random variable ($\rho^{(i)}_{kl},\rho^{(j)}_{kl}$ or $\rho^{(i,j)}_{kl}$), and if this random variable is positive then the appropriate quantity ($\hat{S}^{(i)}_{kl},\hat{S}^{(j)}_{kl}$ or $\hat{D}_{kl}$) is calculated and the sample value is updated according to \eqref{expr:sthetahat}.
\end{itemize}

In this algorithm, we  assume that the function \emph{rand()} returns independent samples from the uniform distribution on $[0,1]$. Moreover $n  \sim \Call{Poisson }  { r}$ implies that $n$ is an independently generated random variable having Poisson distribution with parameter $r$.  When the state of the process is $x$, the next time increment ($\Delta s$) and reaction index ($k$), as prescribed by Gillespie's SSA \cite{GP}, can be calculated using function $\Call{SSA}{x}$ (see Algorithm \ref{SSA}).

\begin{algorithm}[h]                     
\caption{Computes the next time increment ($\Delta s$) and reaction index $(k)$ for Gillespie's SSA}
\label{SSA}  
\begin{algorithmic}[1]
\Function{SSA}{$x$} 
\State Set $r_1 = \mathrm{rand}()$ , $r_2 = \mathrm{rand}()$ and $k = 0$

	\State Calculate $\Delta s= -\log(r_1)/\lambda_0(x,\theta)$
	\State Set $S = 0$
	\While {$S < r_2$}
		\State Update $k \gets k + 1$
		\State Update $S \gets S + \lambda_k(x,\theta)/ \lambda_0(x,\theta) $
	\EndWhile

\State \Return $(\Delta s, k)$
\EndFunction
\end{algorithmic}
\end{algorithm}

\begin{algorithm}[h]
\caption{Generates one realization of $\hat{s}^{(i,j)}_\theta (x_0,f,t) $ according to \eqref{expr:sthetahat} }
\label{gensensvalue}     
\begin{algorithmic}[1]
\Function{GenerateSample}{$x_0,t,c$}
    \State Set $x = x_0$, $s = 0$ and $S = 0$
\While {$ s <  t $} 
\State  Calculate $(\Delta s, k_0 ) =$ SSA$(x)$
\State Update $\Delta s  \gets \min\{\Delta s, t -s\}$ and set $\gamma = -\frac{ \log(rand())}{ \lambda_0(x,\theta) }$
\If{$\gamma \geq (t-s)$ }
\State Update $S \gets S + \left( \frac{ \partial \lambda^2_k(x,\theta) }{ \partial \theta_i \partial \theta_j } \right) (f(x+\zeta_k) -f(x) ) \Delta s$
\Else
\State Update $S \gets S + \left( \frac{ \partial \lambda^2_k(x,\theta) }{ \partial \theta_i \partial \theta_j } \right)  (f(x+\zeta_k) -f(x) ) \left( \Delta s - \frac{ 1 }{ \lambda_0(x,\theta) }  \right)$
\For {$k = 1$ to $K$}	
\State Set $r^{(i)} = \left|  \frac{ \partial \lambda_k(x,\theta) }{ \partial \theta_i }  \right|$,  $r^{(j)} = \left|  \frac{ \partial \lambda_k(x,\theta) }{ \partial \theta_j }  \right|$  and $r^{(i,j)} = \left|  \frac{ \partial^2 \lambda_k(x,\theta) }{ \partial \theta_i \partial \theta_j }  \right|$
\State Set $\beta^{(i)} =\textnormal{Sign}\left( \frac{ \partial \lambda_k(x,\theta) }{ \partial \theta_i }  \right)$,  $\beta^{(j)} = \textnormal{Sign}\left(  \frac{ \partial \lambda_k(x,\theta) }{ \partial \theta_j }  \right)$  and $\beta^{(i,j)} = \textnormal{Sign}\left(  \frac{ \partial^2 \lambda_k(x,\theta) }{ \partial \theta_i \partial \theta_j }  \right)$
\If {$r^{(i)} > 0$}
\State Set $n \sim \Call{Poisson } { \frac {r^{(i)} c }{ \lambda_0(x,\theta) }  }$
\If {$n > 0$}
\State Calculate $\hat{S}^{(i)}_{kl}$ according to \eqref{defn_sqkl} with $q=i, \sigma_l =s, \gamma_l = \gamma$ and $X_\theta(\sigma_l)=x$.
\State Update $S \gets S +\left( \frac{\beta^{(i)} n}{c}  \right) \hat{S}^{(i)}_{kl}$
\EndIf
\EndIf

\If {$r^{(j)} > 0$}
\State Set $n  \sim \Call{Poisson }  { \frac {r^{(j)} c }{ \lambda_0(x,\theta) }  }$
\If {$n > 0$}
\State Calculate $\hat{S}^{(j)}_{kl}$ according to \eqref{defn_sqkl} with $q=j,\sigma_l =s, \gamma_l = \gamma$ and $X_\theta(\sigma_l)=x$.
\State Update $S \gets S +\left( \frac{\beta^{(j)} n}{c}  \right) \hat{S}^{(j)}_{kl}$
\EndIf
\EndIf

\If {$r^{(i,j)} > 0$}
\State Set $n  \sim \Call{Poisson }  { \frac {r^{(i,j)} c }{ \lambda_0(x,\theta) }  }$
\If {$n > 0$}
\State Calculate $\hat{D}_{kl}$ according to \eqref{defn_dkl} with $\sigma_l =s, \gamma_l = \gamma$ and $X_\theta(\sigma_l)=x$.
\State Update $S \gets S +\left( \frac{\beta^{(i,j)} n}{c}  \right) \hat{D}_{kl}$
\EndIf
\EndIf

\EndFor

\EndIf

\State Update $s \gets s +\Delta s$ and $x \gets x+\zeta_{k_0}$
\EndWhile 
\State \Return $S$
\EndFunction
\end{algorithmic}
\end{algorithm}


\begin{thebibliography}{1}

\bibitem{DA}
D.~Anderson.
\newblock An efficient finite difference method for parameter sensitivities of
  continuous time markov chains.
\newblock {\em SIAM: Journal on Numerical Analysis}, 2012.

\bibitem{EK}
S.~N. Ethier and T.~G. Kurtz.
\newblock {\em Markov processes}.
\newblock Wiley Series in Probability and Mathematical Statistics: Probability
  and Mathematical Statistics. John Wiley \& Sons Inc., New York, 1986.
\newblock Characterization and convergence.

\bibitem{GP}
D.~T. Gillespie.
\newblock Exact stochastic simulation of coupled chemical reactions.
\newblock {\em The Journal of Physical Chemistry}, 81(25):2340--2361, 1977.

\bibitem{Gupta}
A.~Gupta and M.~Khammash.
\newblock Unbiased estimation of parameter sensitivities for stochastic
  chemical reaction networks.
\newblock {\em SIAM : Journal on Scientific Computing}, 35(6):2598 -- 2620,
  2013.

\bibitem{Gupta2}
A.~Gupta and M.~Khammash.
\newblock An efficient and unbiased method for sensitivity analysis of
  stochastic reaction networks.
\newblock {\em Unpublished. Available on arXiv: 1402.3076}, 2014.

\bibitem{Gir}
S.~Plyasunov and A.~Arkin.
\newblock Efficient stochastic sensitivity analysis of discrete event systems.
\newblock {\em Journal of Computational Physics}, 221:724--738, 2007.

\bibitem{KSR1}
M.~Rathinam, P.~W. Sheppard, and M.~Khammash.
\newblock Efficient computation of parameter sensitivities of discrete
  stochastic chemical reaction networks.
\newblock {\em Journal of Chemical Physics}, 132, 2010.

\bibitem{KSR2}
P.~W. Sheppard, M.~Rathinam, and M.~Khammash.
\newblock A pathwise derivative approach to the computation of parameter
  sensitivities in discrete stochastic chemical systems.
\newblock {\em Journal of Chemical Physics}, 136, 2012.

\bibitem{Beth}
E.~S. Wolf and D.~F. Anderson.
\newblock A finite difference method for estimating second order parameter
  sensitivities of discrete stochastic chemical reaction networks.
\newblock {\em The Journal of Chemical Physics}, 137(22), 2012.

\end{thebibliography}

\end{document}